\documentclass[a4paper, 11pt]{article}
\usepackage{amsmath}
\usepackage{amssymb}
\usepackage{amsthm}
\usepackage[dvipdfmx]{graphicx}
\usepackage{comment}
\usepackage{cases}
\usepackage{float}
\usepackage[top=25truemm, bottom=25truemm, left=23.5truemm, right=23truemm]{geometry}
\usepackage{titlefoot}

\allowdisplaybreaks[1]

\newcommand{\divergence}{\mathrm{div}}

\newcommand{\dist}{\mathrm{dist}}

\let\OLDthebibliography\thebibliography
\renewcommand\thebibliography[1]{
	\OLDthebibliography{#1}
	\setlength{\parskip}{1pt}
	\setlength{\itemsep}{1pt plus 0.3ex}
}

\makeatletter

\@addtoreset{equation}{section}
\makeatother

\newtheoremstyle{mystyle}
{}
{}
{\itshape}
{}
{\bfseries}
{}
{ }
{}

\begin{document}

\theoremstyle{mystyle}
\newtheorem{theorem}{Theorem}[section]
\newtheorem*{theorem*}{Theorem}
\newtheorem{lemma}[theorem]{Lemma}
\newtheorem{proposition}[theorem]{Proposition}
\newtheorem{corollary}[theorem]{Corollary}
\newtheorem{prob}[theorem]{Problem}
\newtheorem{aim}[theorem]{Aim}
\theoremstyle{definition}  
\newtheorem{definition}[theorem]{Definition}
\theoremstyle{remark}
\newtheorem{remark}[theorem]{Remark}

\renewcommand{\thefootnote}{\fnsymbol{footnote}}

\title{Nonexistence of Minimizers for a Nonlocal Perimeter Functional with a Riesz and a Background Potential}

\author{
Fumihiko Onoue\footnote{Scuola Normale Superiore, Piazza Cavalieri 7, 56126 Pisa, Italy. E-mail: {\tt fumihiko.onoue@sns.it}}
}


\maketitle

\begin{abstract}
We consider the nonexistence of minimizers for the energy containing a nonlocal perimeter with a general kernel $K$, a Riesz potential, and a background potential in $\mathbb{R}^N$ with $N\geq2$ under the volume constraint. We show that the energy has no minimizer for a sufficiently large mass under suitable assumptions on $K$. The proof is based on the partition of a minimizer and the comparison of the sum of the energy for each part with the energy for the original configuration. This strategy is shown in \cite{FKN} and \cite{lamanna1}.
\end{abstract}

\section{Introduction}
We study the minimizing problem for the following fuctional:
\begin{equation}\label{nonlocalFunc2}
	\mathcal{F}_{(K,\,A)}(E):=P_{K}(E)+V_1(E)-A\,R(E)
\end{equation}
under the volume constraint $|E|=m$ where $E\subset\mathbb{R}^N$, $N\geq 2$, is measurable, and $A\geq0$. Here $P_K(E)$ is the nonlocal perimeter with a general kernel $K$ defined by
\begin{equation}\label{nonlocalPeri}
	P_K(E):= \displaystyle\int_{E}\int_{E^c} K(x-y)\,dx\,dy,
\end{equation}
where $E^c:= \mathbb{R}^N\setminus E$ and the kernel $K$ satisfies the following conditions:
\begin{itemize}
	\item[(K1)] $K(x)\geq 0$, and $K(-x)=K(x)$ for any $x\in\mathbb{R}^N$.
	\item[(K2)] $K\in L^1(B^c_1(0))$.
	\item[(K3)] $K$ is a locally Lipschitz function in $\mathbb{R}^{N}\setminus\{0\}$.
\end{itemize}
Moreover, in our study, we further impose either of the following assumptions: 
\begin{itemize}
	\item[(K4)] there exists constants $0<s<1$ and $\varepsilon>2^{\frac{1}{N+s-1}}-1>0$ such that, for any $0<|x|<1+\varepsilon$, $K(x)\leq |x|^{-(N+s)}$ and, for any $|x|\geq 1+\varepsilon$, $|x|K(x) \leq (1+\varepsilon)^{-(N+s-1)}$.
	\item [(K4)'] there exists constants $0<s<1$, $\varepsilon>2^{\frac{1}{N+s-1}}-1>0$, and $\lambda>1$ such that $|x|^{-(N+s)}\leq K(x)\leq \lambda|x|^{-(N+s)}$ for any $0<|x|<1+\varepsilon$ and $|x|^{-(N+s)}\leq K(x)\leq (1+\varepsilon)^{-(N+s)}$ for any $|x|\geq 1+\varepsilon$.
\end{itemize}
For instance, $K(x):= |x|^{-(N+s)}$ is one of the examples satisfying all of the above conditions on $K$. From the above conditions, we can roughly illustrate the shape of $K$ in Figure \ref{figure1} and \ref{figure2}.

Moreover, $V_1$ is the Riesz potential and is defined by
\begin{equation}\label{coulombicPotential}
	V_1(E):=\frac{1}{2}\int_{E}\int_{E} \frac{1}{|x-y|}\,dx\,dy,
\end{equation}
for any measurable $E\subset\mathbb{R}^{N}$. Specifically, in the case of $N=3$, $V_1$ is said to be the Coulombic repulsive potential. Finally, $R(E)$ is defined by
\begin{equation}\label{thirdPotential}
	R(E):=\int_{E}\frac{1}{|x|}\,dx,
\end{equation}
for any measureble $E\subset\mathbb{R}^N$. 

\begin{figure}[H]
	\begin{center}
		\begin{tabular}{c}
			\begin{minipage}{0.50\hsize}
				\includegraphics[keepaspectratio,scale=0.54]{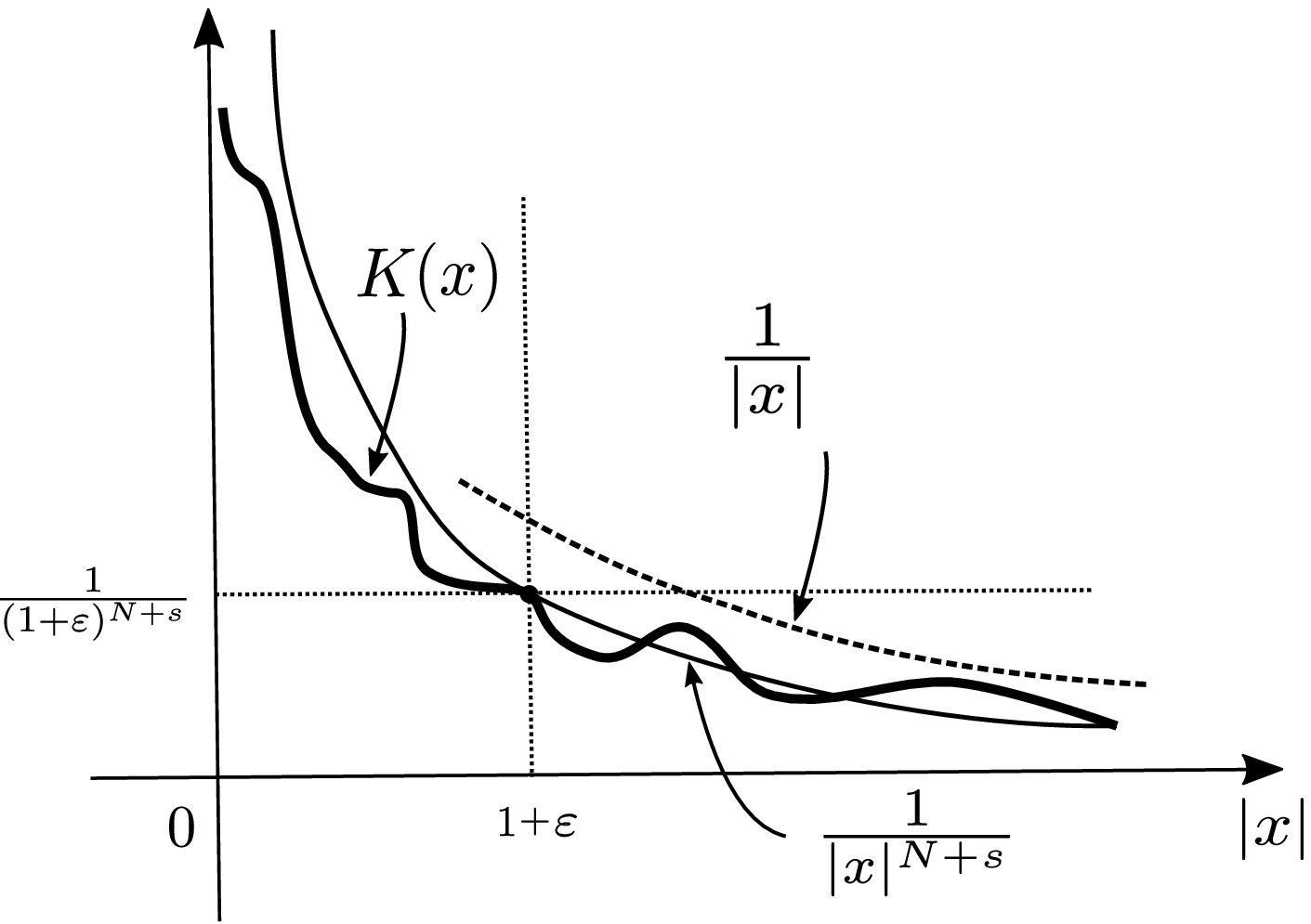}
				\caption{Graph of $K$ with (K4)}
				\label{figure1}
			\end{minipage}
		
			\begin{minipage}{0.50\hsize}
				\includegraphics[keepaspectratio,scale=0.54]{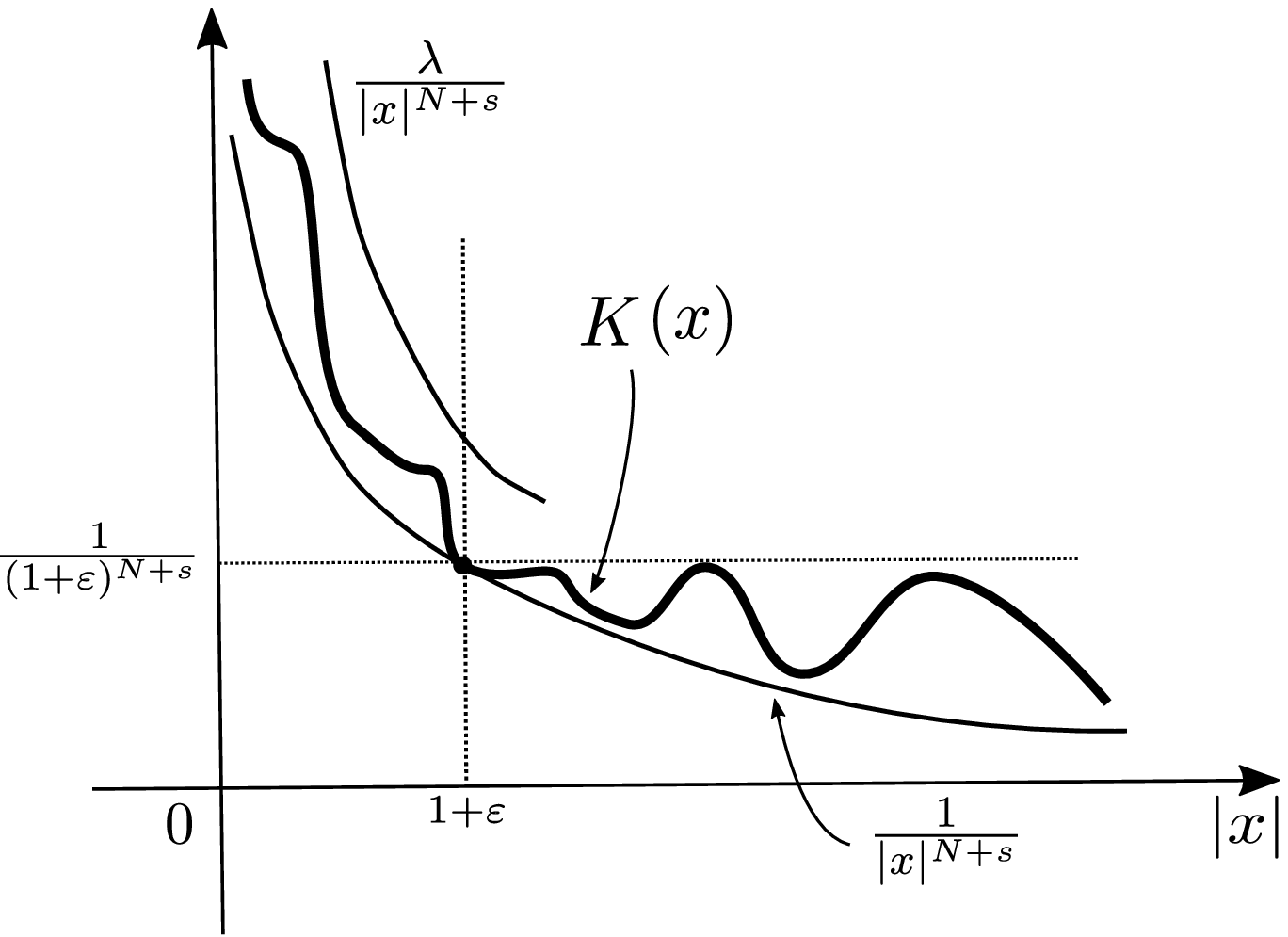}
				\caption{Graph of $K$ with (K4)'}
				\label{figure2}
			\end{minipage}
		\end{tabular}
	\end{center}
\end{figure}

The study of the variational problem of \eqref{nonlocalFunc2} is one extension of a series of the previous works \cite{julin, knupferMuratov, lamanna2, luOtto}, in which they treated the classical perimeter instead of the nonlocal one. Especially, when $N=3$, $\alpha=2$, and $A\equiv0$, it is important to consider this minimizing problem in physics. It is known as the liquid drop model, introduced by Gamow \cite{gamow} to model the stability of atomic nuclei and nuclear fission. In this paper, we consider the nonexistence of the minimizers for \eqref{nonlocalFunc2} all over the measurable sets in $\mathbb{R}^N$ with the volume constraint. Note that, considering the general Riesz potential $V_{\alpha}$, instead of $V_1$, which is defined by
\begin{equation}\label{3dimRieszPo}
	V_{\alpha}(E):=\frac{1}{2}\int_{E}\int_{E}\frac{1}{|x-y|^{\alpha}}\,dx\,dy
\end{equation}
for any $E\subset\mathbb{R}^N$ and some $\alpha\in(0,\,N)$, we are not able to show the nonexistence of minimizers so far. This is because we do not overcome the technical difficulties in estimating $V_{\alpha}$ to obtain the claim (see the proof of Theorem \ref{mainTheorem} for details).

If we consider each term in \eqref{nonlocalFunc2} separately and we choose a function $|x|^{-(N+s)}$ as the kernel $K$, then it is known that, by the isoperimetric inequality for the nolocal perimeter, a ball is the only minimizer for $P_K$ all over the sets with finite volume and, by Riesz rearrangement inequality, a ball is the only maximizer for both $V_{\alpha}$ and $R$ under the volume constraint. Thus, it is not trivial whether the minimizing problem of \eqref{nonlocalFunc2} has a solution under the volume constraint or not. Moreover, we may not expect the existence of a minimizer for \eqref{nonlocalFunc2} if $m$ is large. Indeed, considering the rescaled set $\lambda\,E$, we observe
\begin{equation}\label{rescaledFunc}
	\mathcal{F}_{(K,\,A)}(\lambda\,E)= \lambda^{N-s}\,P_K(E)+ \lambda^{2N-1}\,V_1(E)- \lambda^{N-1}A\,R(E).
\end{equation}
Thus, when $\lambda$ is large, the dominating term is $\lambda^{2N-1}\,V_1(E)$, however $V_1$ does not admit minimizers. Note that, if $\lambda$ is sufficiently small, then the dominating term is $\lambda^{N-s}\,P_K(E)$ because of $s<1$. Especially, in the case that $K(x)=|x|^{-(N+s)}$, it is known that $P_K$ admits a minimizer and it is actually a ball. This implies that, if $\lambda$ is small, then there may exist a minimizer for \eqref{nonlocalFunc2} and it should be a ball in $\mathbb{R}^N$.

Now let us briefly review the previous works concerned with our problem. In the case of the classical perimeter and the general Riesz potential $V_{\alpha}$ instead of $P_K$ and $V_1$ in \eqref{nonlocalFunc2}, Kn\"upfer and Muratov proved a series of the results stated in \cite{knupferMuratov0, knupferMuratov} as follows; if $N=2$, $\alpha\in(0,\,2)$, and $A\equiv 0$, balls are the only minimizer under the volume constraint $|E|=m$ for sufficiently small $m>0$. In addition, for sufficiently large $m>0$, there are no minimizers. Finally, in higher dimensions, if $3\leq N \leq 7$ and $\alpha\in(0,\,N-1)$, then the balls are the only minimizers for sufficiently small $m>0$. In addition, Bonacini and Cristoferi extended some of the above results to the case $N\geq 3$ in \cite{BC}. Especially, regarding to the nonexistence of minimizers, not only Kn\"upfer and Muratov but also Lu and Otto showed in \cite{luOtto} the following result; if $N=3$ and $A\not\equiv 0$, then there exists a number $m_0>0$, which is explictly obtained, that for any $m\geq m_0$, the minimizing problem $\min\{P(E)+V_1(E)-A\,R(E)\mid |E|=m\}$ has no solution. In the case of the nonlocal perimeter, Figalli et al. in \cite{FFMMM} considered the isoperimetric problems with $K(x)=|x|^{-(N+s)}$ in the presence of the Riesz potential if $N\geq 2$. Precisely, they showed that, if $m$ is positive number less than some number $m_0$, then the problems 
\begin{equation}\label{FFMMMpaper}
	\inf\left\{\frac{1-s}{\omega_{N-1}}\,P_s(E)+V_{N-\alpha}(E)\mid |E|=m\right\},\quad P_s(E):=\int_{E}\int_{E^c}\frac{1}{|x-y|^{N+s}}\,dx\,dy,
\end{equation}
where $\omega_{N-1}$ is the volume of the unit sphere $\mathbb{S}^{N-1}$, admit balls of volume $m$ as their (unique up to translation) minimizers for any $s\in(s_0,\,1)$ and $\alpha\in(\alpha_0,\,N)$ with some uniform lower bounds $s_0,\,\alpha_0>0$.

In this paper, we established the general results of the nonexistence of the minimizers in the case of nonlocal perimeter. Precisely, we obtain the following two results; the first one is that, if we assume that $K$ satisfies (K1), (K2), (K3), and (K4)' and assume $A\geq 0$ and $m>0$, then every minimizer $E$ of $\mathcal{F}_{(K,\,A)}$ is bounded. The second one is that, if we assume that $K$ satisfies (K1), (K2), and (K4) and $m>m_0(N,\,s,\,\varepsilon),\,A$, then $\min\{\mathcal{F}_{(K,\,A)}(E)\mid |E|=m\}$ has no bounded solutions. Accordingly, if we assume that the conditions (K1), (K2), (K3), and (K4)' are valid, then we may have that $\min\{\mathcal{F}_{(K,\,A)}(E)\mid |E|=m\}$ has no solutions. We emphsize that our result is one complementing result shown in \cite{FFMMM}. Indeed, if we set $A\equiv0$ and $K(x)= |x|^{-(N+s)}$, then our result gives us the solution to the nonexistence problem for the functional $P_s(E)+V_1(E)$ under the volume constraint $|E|=m$, while the authors in \cite{FFMMM} considered the existence of minimizers for the functional $P_s(E)+V_{\alpha}(E)$ where $\alpha\in(\alpha_0,\,N)$ for some $\alpha_0>0$ under the volume constraint.

Our idea for proving the boundedness of minimizers is based on \cite{CN}. First, we will show that the energy \eqref{nonlocalFunc2} is stable under some sufficiently small peturbation near a point in the measure-theoretic boundary of a minimizer. This stability is what we call ``Almgren's lemma". Secondly, if we suppose that a minimizer is not bounded, this stability and the minimality make it possible that one can derive an integral inequality of the volume of a minimizer outside of some ball, which leads to a contradiction. On the other hand, the idea for proving the main theorem, i.e., the nonexistence of minimizers is based on \cite{julin, luOtto2, FKN, lamanna1}. First of all, we recall that, if $K(x)=|x|^{-(N+s)}$, then Kn\"upfer and Muratov showed in \cite{knupferMuratov0, knupferMuratov} the nonexistence of minimizers by the following idea; first, they showed that the minimizers with large volumes must be long and thin; otherwise the Riesz potential can gain the energy. Thus, they cut the minimizers for large masses into two parts by some hyperplane. Then, they moved the two parts far away from each other to reduce the energy. As a result, they can compared the energy of the sum of the two sets with that of the original configuration, to find that it can be represented as a differential inequality for the volumes of a cross-section of the minimizer by different hyperplanes. In our case, however, due to the last term of \eqref{nonlocalFunc2}, even if we cut the minimizers into two parts and move them far away from each other, there is a possibility that one may decrease the total energy. Thus, the idea in \cite{knupferMuratov0, knupferMuratov} might not work in our case. Instead, we will take the following strategy; first, we separate $\mathbb{R}^N$ into two parts by a hyperplane which is parametrized by a directional parameter $\nu\in\mathbb{S}^{N-1}$ and a tranlating parameter $l\in\mathbb{R}$. Then, taking any minimzer of \eqref{nonlocalFunc2} and considering the intersection of the minimizer by each separated part (either of them can be an empty set), we compare the sum of the energy for each intersetion with that of the original set. Integrating the resulitng inequality with respect to $l$ and then $\nu$, we can obtain the inequality of the mass of a minimizer to find that it actually shows the upper bound of the mass.

This paper is divided into four sections. In Section \ref{statementResults}, we state the two results, one is on the boundedness of minimizers and the other is on the nonexitence of minimizers. In Section \ref{boundednessMinimizer}, we first show the boundedness of minimziers under the assumptions (K1), (K2), (K3), and (K4)' on the kernel $K(x)$. In Section \ref{nonexisMinimizer}, we next prove the nonexistence of minimizers. Before proving the nonexitence, we first show the weak subadditivity of the infimum of the energy \eqref{nonlocalFunc2} with respect to the mass $m>0$. Finally, in Section \ref{appendix}, we state one possible generalization of the energy and the proof of the nonexistence of minimizers for it. This is obtained by a simple modification of the proof done in Section \ref{nonexisMinimizer}.

\section{Statement of results}\label{statementResults}
Our main results in this paper is as follows; the first one is on the boundedness of the minimizer and the second one is on the nonexistence of the minimizers.
\begin{lemma}\label{mainLemma}
	Suppose that $N\geq 2$ and $K$ satisfies (K1), (K2), (K3), and (K4)' and let $A\geq 0$ and $m\in(0,\,\infty)$. Then, for every minimizer  $E$ of $\mathcal{F}_{(K,\,A)}$, $E$ is bounded.
\end{lemma}

\begin{theorem}\label{mainTheorem}
	Suppose that $N\geq 2$ and $K$ satisfies (K1), (K2), and (K4) and let $A\geq 0$. Then, there exists a critical mass $m_c>0$ given by
	\begin{equation}\label{volumeCondition}
		m_c:=\left(\frac{1}{2}-\frac{1}{(1+\varepsilon)^{N+s-1}}\right)^{-1} \left(\frac{\omega_{N-1}}{1-s}\,(1+\varepsilon)^{1-s} + A\right)
	\end{equation}
	such that, for any $m>m_c$, 
	\begin{equation}\label{minimizerFunctional}
		\min\{\mathcal{F}_{(K,\,A)}(E)\mid E\subset\mathbb{R}^N \text{measureble},\,|E|=m\}
	\end{equation}
	has no bounded solutions.
	
	In addition, if we impose the conditions (K3) and (K4)' instead of (K4) on $K$, then, for any $m>m_c$ where $m_c$ is the same critical mass as \eqref{volumeCondition}, \eqref{minimizerFunctional} has no solutions.
\end{theorem}
The proof is based on \cite{FKN, lamanna1}, however, we need to pay more attention on the calculation of the nonlocal perimeter in order to obtain the explicit lower bound of the mass.

\section{Boundedness of minimizers}\label{boundednessMinimizer}
In this section, we prove Lemma \eqref{mainLemma}, i.e., the boundedness of minimizers for the functional \eqref{nonlocalFunc2} under some conditions on a kernel $K$. First of all, we start to show a generalized version of so-called Almgren's lemma (see \cite{frank} for the classical results, \cite{FM} in the case of an anisotropic perimeter, or \cite{CN} in the case of a nonlocal s-perimeter).
\begin{lemma}\label{almgrenLemma}
	Suppose that $K$ satisfies (K1), (K2), and (K3). Let $E\subset\mathbb{R}^N$ be a measurable set with $P_{K}(E)+V_1(E)+R(E)<\infty$. Let $x_0\in\mathbb{R}^N$ and $r_0>0$ be such that
	\begin{equation}\label{densityassumptions}
	|B_{r_0}(x_0)\cap E| >0, \quad |B_{r_0}(x_0)\cap E^c| >0.
	\end{equation}
	Then, there exist $k_0,\,C>0$ such that, for any $k\in(-k_0,\,k_0)$, there exists a measurable set $F_k\subset\mathbb{R}^N$ such that $P_{K}(F_k)+V_1(F_k)+R(F_k)<\infty$ and the following properties hold:
	\begin{enumerate}
		\item $E\varDelta F_k \Subset B_{r_0}(x_0)$
		\item $|F_k| - |E| = k$
		\item $|\mathcal{F}_{(K,\,A)}(F_k)-\mathcal{F}_{(K,\,A)}(E)| < C\,k$.
	\end{enumerate}
\end{lemma}
\begin{proof}\label{pfAlmgrenLem}
	From the density assumption, we have that there exists a function $T\in C^{1}_c(B_{r_0}(x_0);\,\mathbb{R}^N)$ such that
	\begin{equation}\label{pfAlmgren1}
	M := \int_{E} \divergence T(x)\,dx > 0.
	\end{equation}
	Otherwise, by the definition of the classical perimeter, it holds that $P(E,\,B_{r_0}(x_0)) = 0$. However, by the classical isoperimetric inequality, we have $|E \cap  B_{r_0}(x_0)| =0$, which contradicts the assumption of $E$.
	
	For any $t\in(-1,\,1)$, we define the maps $\Psi_t(x):= x + t\,T(x)$ for all $x\in\mathbb{R}^N$. Then, we may easily see that there exists $\delta_0\in(0,\,1)$ such that the maps $\Psi_t$ are diffeomorphisms from $\mathbb{R}^N$ onto itself for any $t\in(-\delta_0,\,\delta_0)$. Moreover, we have that $\det(\nabla\Psi_t(x)) = 1 + t\divergence T(x) + o(t)$ for any $t\in(-\delta_0,\,\delta_0)$. By the definition of $\Psi_t$, we also have $E\Delta\Psi_t(E)\subset\subset\mathbb{R}^N$ and thus, applying the change of variables,
	\begin{equation}\label{pfAlmgren2}
	|\Psi_t(E)| = \int_{E}|\det\nabla\Psi_t(x)|\,dx= \int_{E}\left(1 + t\divergence T(x) + o(t)\right)\,dx = |E| + tM + o(t),
	\end{equation}
	for sufficiently small $t\in(-\delta_0,\,\delta_0)$. Therefore, there exists a constant $k_0>0$ such that, if we set $F_k:= \Psi_{t(k)}(E)$, where $t(k):=k/M+o(k)$, for any $k\in(-k_0,\,k_0)$, $F_k$ satisfies the first and second properties in Lemma \ref{almgrenLemma}.
	
	Now, from the locally Lipschitz property of $K$ and the compactness of the support of $T$, we have that there exists a constant $C_0:=C_0(T)>0$ such that $K(x-y+t(T(x)-T(y)))\leq K(x-y) + C_0\,|t||T(x)-T(y)|$ for any $x,\,y\in B_{r_0}(x_0),\,x\neq y$ and any $0<t<1$. Then, we may compute $|\mathcal{F}_{(K,\,A)}(F_k)-\mathcal{F}_{(K,\,A)}(E)|$ as follows: first, we have
	\begin{align}\label{pfAlmgren2.1}
		&\Big|K(\Psi_t(x)-\Psi_t(y))\,|\det\nabla\Psi_t(x)|\,|\det\nabla\Psi_t(y)|-K(x-y)\Big| \nonumber\\
		&=\Big|K(x-y+t(T(x)-T(y)))\,\left(1+t(\divergence T(x)+\divergence T(y))+o(t)\right)-K(x-y)\Big|\nonumber\\
		&\leq |t|(2\|T\|_{C^1}+|o(1)|)\,K(x-y)+ C_0|t||T(x)-T(y)|(1+2|t|\|T\|_{C^1}+|o(t)|),
	\end{align}
	for any $x,\,y\in B_{r_0}(x_0)$ with $x\neq y$. Since the left-hand side of \eqref{pfAlmgren2.1} is equal to 0 if $x,\,y \notin B_{r_0}(x_0)$, it holds that
	\begin{align}\label{pfAlmgren3}
	&|P_{K}(\Psi_t(E))-P_{K}(E)|\nonumber\\
	&\leq \int_{E}\int_{E^c}\Big|K(\Psi_t(x)-\Psi_t(y))\,|\det\nabla\Psi_t(x)|\,|\det\nabla\Psi_t(y)|-K(x-y)\Big|\,dx\,dy \nonumber\\
	&\leq |t|\int_{E\cap B_{r_0}(x_0)}\int_{E^c\cap B_{r_0}(x_0)}\,\Big(C(T)\,K(x-y)+C'_0(T)\Big) \,dx\,dy \nonumber\\
	&\leq |t|\Big(C(T)P_{K}(E)+C'_0(T)|B_{r_0}(x_0)|^2\Big)<\infty,
	\end{align}
	where $C(T),\,C'_0(T)>0$ are constants independent of $0<t<1$. Secondly, we also have that
	\begin{align}\label{pfAlmgren4}
		|V_1(\Psi_t(E))-V_1(E)|&\leq \int_{E}\int_{E} \left|\frac{|\det\nabla\Psi_t(x)|\,|\det\nabla\Psi_t(y)|}{|\Psi_t(x)-\Psi_t(y)|}-\frac{1}{|x-y|}\right|\,dx\,dy \nonumber\\
		&= \int_{E}\int_{E} \left|\frac{1+t(\divergence T(x)+\divergence T(y))+o(t)}{|x-y+t(T(x)-T(y))|}-\frac{1}{|x-y|}\right|\,dx\,dy \nonumber\\
		&\leq \int_{E}\int_{E}\left(\frac{C_1(T)|t|}{|x-y|}+ \frac{o(t)}{|x-y|}\right)\,dx\,dy \nonumber\\
		&\leq |t|\,C_2(T)\,V_1(E) <\infty,
	\end{align}
	where $C_1(T),\,C_2(T)$ are positive constants independent of $t$. Finally, we estimate the difference of the potentials $R(\Psi_t(E))$ and $R(E)$. Since $t$ is sufficiently small, there exists a constant $c_0(T)>0$ such that $|x|(1-c_0(T)\,|t|) \leq |x+t\,T(x)| \leq |x|(1+c_0(T)\,|t|)$ and we have that $(1\pm c_0(T)\,|t|)=1\pm\,c_0(T)\,|t|+o(t)$. Hence, we obtain
	\begin{align}\label{pfAlmgren5}
		|R(\Psi_t(E))-R(E)|\leq \int_{E}\left|\frac{1+t\,\divergence T(x)+ o(t)}{|x+t\,T(x)|}-\frac{1}{|x|}\right|\,dx &\leq \int_{E} \left(\frac{C_3(T)|t|}{|x|}+\frac{o(t)}{|x|}\right)\,dx\nonumber\\
		&\leq |t|\,C_4(T)\,R(E)<\infty
	\end{align}
	where $C_3(T),\,C_4(T)$ are positive constants independent of $t$. Therefore, \eqref{pfAlmgren3}, \eqref{pfAlmgren4}, and \eqref{pfAlmgren5} imply that the set $F_k$ satisfies the third property in Lemma \ref{almgrenLemma} and this completes the proof.
\end{proof}

Before starting to prove Lemma \ref{mainLemma}, we need the following isoperimetric inequality of the nonlocal perimeter $P_K$ for sets with small volumes if $K$ satisfies the conditions (K1) and (K4)'.
\begin{lemma}\label{isoperiInequalityKernel}
	Suppose that $K$ satisfies (K1) and (K4)'. Then, there exists a constant $C=C(N,\,s,\,\lambda)>0$ such that, for any measurable set $F\subset\mathbb{R}^N$ with $0<|F|\leq \omega_N(1+\varepsilon)^N$, we have
	\begin{equation}\label{lemma_isoperi_1}
	C\,|F|^{\frac{N-s}{N}}\leq P_{K}(F). 
	\end{equation}
\end{lemma}
\begin{proof}
	If $P_{K}(F)=\infty$, then the lemma is proved. Thus, we can assume that $P_K(F)$ is finite. Setting $\rho:=\omega_N^{-1/N}|F|^{1/N}\in(0,\,1+\varepsilon]$, we may compute the nonlocal perimeter $P_{K}(F)$ as follows:
	\begin{align}\label{pfIsoperi1}
	\int_{F}\int_{F^c}K(x-y)\,dx\,dy & = \int_{F}\int_{F^c\cap B_{\rho}(y)}K(x-y)\,dx\,dy + \int_{F}\int_{F^c\cap (B_{\rho}(y))^c}K(x-y)\,dx\,dy \nonumber\\
	&\geq \frac{1}{\rho^{N+s}}\int_{F}|F^c\cap B_{\rho}(y)|\,dy + \int_{F}\int_{F^c\cap (B_{\rho}(y))^c}K(x-y)\,dx\,dy.
	\end{align}
	Here, by the assumption of $F$, we have $|F|=\omega_N\rho^N=|B_{\rho}(y)|$ for any $y$. Thus, it holds that
	\begin{align}\label{pfIsoperi2}
	|A^c\cap B_{\rho}(y)| &=|B_{\rho}(y)|- |A\cap B_{\rho}(y)|\nonumber\\
	&= |A| - |A\cap B_{\rho}(y)| = |A\cap (B_{\rho}(y))^c|
	\end{align}
	for any $y\in F$. Thus, from \eqref{pfIsoperi1} and \eqref{pfIsoperi2} and recalling the assumption of $K$ and the fact that $\lambda>1$ and $\rho<1+\varepsilon$, we obtain the following inequalty:
	\begin{align}\label{pfIsoperi3}
	\int_{F}\int_{F^c}K(x-y)\,dx\,dy &\geq  \frac{1}{\rho^{N+s}}\int_{F}|F\cap (B_{\rho}(y))^c|\,dy + \int_{F}\int_{F^c\cap (B_{\rho}(y))^c}K(x-y)\,dx\,dy \nonumber\\
	&\geq  \frac{1}{\lambda}\int_{F}\int_{F\cap(B_{\rho}(y))^c}K(x-y)\,dx\,dy +\int_{F}\int_{F^c\cap (B_{\rho}(y))^c}K(x-y)\,dx\,dy\nonumber\\
	&\geq \frac{1}{\lambda} \int_{F}\int_{(B_{\rho}(y))^c}K(x-y)\,dx\,dy\nonumber\\
	&\geq \frac{1}{\lambda} \int_{F}\int_{(B_{\rho}(y))^c}\frac{1}{|x-y|^{N+s}}\,dx\,dy.
	\end{align}
	Moreover, by applying the change of variables, we can further calculate the last term in \eqref{pfIsoperi3} in the following manner:
	\begin{align}\label{pfIsoperi4}
	\int_{F}\int_{(B_{\rho}(y))^c}\frac{1}{|x-y|^{N+s}}\,dx\,dy &= |F|\,\omega_{N-1}\,\int_{\rho}^{\infty}\frac{1}{r^{1+s}}\,dx\,dy \nonumber\\
	&= |F|\frac{\omega_{N-1}}{s}\rho^{-s} = \frac{\omega_{N-1}\omega_N^{\frac{s}{N}}}{s}|F|^{\frac{N-s}{N}}.
	\end{align}
	Therefore, from \eqref{pfIsoperi3} and \eqref{pfIsoperi4}we obtain
	\begin{equation}\label{pfIsoperi5}
	P_{K}(F)\geq \frac{\omega_{N-1}\omega_N^{\frac{s}{N}}}{\lambda\,s}|F|^{\frac{N-s}{N}}
	\end{equation}
	which is a required inequality.
\end{proof}

Now we are prepared to prove Lemma \ref{mainLemma} by applying the above two claims. The proof is based on the strategy shown in \cite{CN} for instance.
\begin{proposition}\label{boundMinimizer}
	Suppose that $K$ satisfies (K1), (K2), (K3), and (K4)'. Every solution $E$ of
	\begin{equation}\label{propBound}
		\min\{\mathcal{F}_{(K,\,A)}(E)\mid E\subset \mathbb{R}^N\,\text{measurable},\,|E|=m\}
	\end{equation}
	is bounded.
\end{proposition}
\begin{proof}
	Let $E$ be a minimizer of $\min_{|E|}\mathcal{F}_{(K,\,A)}(E)$. For any $r>0$, we define $f(r):=|E\setminus B_r(0)|$. Then, by the continuity of measure and $|E|=m$, $f$ is a non-increasing function and converges to zero as $r\to\infty$. Moreover, the coarea formula implies 
	\begin{equation}\label{prop_1}
	f^{\prime}(r)= - \mathcal{H}^{N-1}(E\cap \partial B_r(0)).
	\end{equation}
	We now show that there exists $R>0$ such that $f(r)=0$ for all $r>R$. This implies the boundedness of minimizers because $\mathcal{F}_{(K,\,A)}(E)$ coincides with $\mathcal{F}_{(K,\,A)}(E')$ if $E\setminus E'$ is a set of Lebesgue measure zero and thus we can identify $E$ with $E'$. Suppose by contradiction that $f(r)>0$ for any $r>0$. Without loss of generality, we can assume that $|E\cap B_1(0)|>0$ and $|\mathcal{C}E\cap B_1(0)|>0$. Choosing $k_0$ as in Lemma \ref{almgrenLemma}, we may fix $R_0>0$ such that $f(r)<k_0$ for any $r\geq R_0$. Then, by Lemma \ref{almgrenLemma}, for any $r\geq R_0$, there exists $F_r\subset\mathbb{R}^N$ such that the followings are true:
	\begin{enumerate}
		\item $E\Delta F_r \subset\subset B_1(0) \subset B_r(0)$.
		\item $|F_r| = |E| + f(r)$.
		\item $|\mathcal{F}_{(K,\,A)}(E)-\mathcal{F}_{(K,\,A)}(F_r)| < C\,f(r)$ for some constant $C>0$ independent of $r>0$.
	\end{enumerate}
	Now, letting $G_r:=F_r\cap B_r(0)$ and recalling the first and second properties of Lemma \ref{almgrenLemma}, we have that $|G_r| = |E|$. Here, in the same way as in \cite{CNRV} for instance, one can prove the equality for a nonlocal perimeter
	\begin{equation}\label{nonlocalPeriEquali}
	P_K(U) + P_K(W) = P_K(U\cup W) + 2\int_{U}\int_{W} K(x-y)\,dx\,dy.
	\end{equation}
	for any measurable sets $U,\,W\subset\mathbb{R}^N$ such that $|U\cap W|=0$. Moreover, we also have the equality for a potential functional
	\begin{equation}\label{coulombPotEquali}
	V_1(U\cup W) = V_1(U) + V_1(W) + \int_{U}\int_{W}\frac{1}{|x-y|}\,dx\,dy.
	\end{equation}
	for any measurable $U,\,W\subset\mathbb{R}^N$ such that $|U\cap W|=0$. Therefore, by the minimality of $E$ and using \eqref{nonlocalPeriEquali} and \eqref{coulombPotEquali}, we obtain
	\begin{align}\label{pfBound1}
	\mathcal{F}_{(K,\,A)}(E)\leq \mathcal{F}_{(K,\,A)}(G_r) &= P_{K}(G_r) + V_1(G_r) - A\,R(G_r)\nonumber\\
	&\leq P_K(F_r)- P_{K}(F_r\setminus B_r(0)) + V_1(F_r)- V_1(F_r\setminus B_r(0)) \nonumber\\
	&\quad \qquad -A\,R(F_r) + A\,R(F_r\setminus B_r(0))+2\int_{F_r\setminus B_r}\int_{F_r\cap B_r}K(x-y)\,dx\,dy \nonumber\\
	&\leq \mathcal{F}_{(K,\,A)}(E) + C\,f(r) - P_{K}(F_r\setminus B_r(0)) \nonumber\\
	&\quad\qquad +A\int_{F_r\setminus B_r}\frac{1}{|x|}\,dx + 2\int_{F_r\setminus B_r}\int_{F_r\cap B_r}K(x-y)\,dx\,dy. 
	\end{align}
	Note that we also used the third property of Lemma \ref{almgrenLemma} in the last inequality of \eqref{pfBound1}. Since $E\setminus B_r(0)= F_r\setminus B_r(0)$, it holds that
	\begin{align}\label{pfBound1.1}
		\int_{F_r\setminus B_r}\frac{1}{|x|}\,dx = \int_{E\setminus B_r}\frac{1}{|x|}\,dx \leq \frac{1}{r}\,|E\setminus B_r| = \frac{f(r)}{r}.
	\end{align}
	By the assumption of $K$ and using the coarea formula, we have that
	\begin{align}\label{pfBound2}
	\int_{F_r\setminus B_r}\int_{F_r\cap B_r}K(x-y)\,dx\,dy &= \int_{E\setminus B_r}\int_{F_r\cap B_r\cap B_{1+\varepsilon}(y)}K(x-y)\,dx\,dy \nonumber\\
	&\quad \qquad +\int_{E\setminus B_r}\int_{F_r\cap B_r\cap B^{c}_{1+\varepsilon}(y)}K(x-y)\,dx\,dy \nonumber\\
	&\leq \int_{E\setminus B_r}\int_{B_r\cap B_{1+\varepsilon}(y)}K(x-y)\,dx\,dy \nonumber\\
	&\quad\qquad +\int_{E\setminus B_r}\int_{F_r\cap B^c_{1+\varepsilon}(y)}K(x-y)\,dx\,dy \nonumber\\
	&\leq \lambda\int_{E\setminus B_r}\int_{B_{|y|-r}(y)}\frac{1}{|x-y|^{N+s}}\,dx\,dy \nonumber\\
	&\quad\qquad + \frac{1}{(1+\varepsilon)^{N+s}}\int_{E\setminus B_r}|F_r|\,dy \nonumber\\
	&\leq \frac{\lambda\omega_{N-1}}{s}\int_{E\setminus B_r} \frac{1}{(|y|-r)^s}\,dy + \frac{1}{(1+\varepsilon)^{N+s}}f(r)\,|E| \nonumber\\
	&\leq \frac{\lambda\omega_{N-1}}{s}\int_{r}^{\infty} \frac{1}{(t-r)^s}\,\mathcal{H}^{N-1}(E\cap \partial B_t(0))\,dt+ \frac{m}{(1+\varepsilon)^{N+s}}f(r)\nonumber\\
	&= - \frac{\lambda\omega_{N-1}}{s}\int_{r}^{\infty}\frac{1}{(t-r)^s}f'(t)\,dt+ \frac{m}{(1+\varepsilon)^{N+s}}f(r).
	\end{align}
	Substituting \eqref{pfBound2} in \eqref{pfBound1}, we obtain
	\begin{equation}\label{pfBound3}
	P_{K}(E\setminus B_r(0)) \leq C_7\,f(r) - \frac{\lambda\omega_{N-1}}{s}\int_{r}^{\infty}\frac{1}{(t-r)^s}f'(t)\,dt
	\end{equation}
	for some constant $C_7>0$ independent of $r>0$. For sufficiently large $r>0$, we can assume that $|E\setminus B_r(0)| < \omega_N(1+\varepsilon)^N$ and thus, by the isoperimetric inequality stated in Lemma \ref{isoperiInequalityKernel}, it holds that 
	\begin{equation}\label{pfBound4}
	C_8\,|E\setminus B_r(0)|^{\frac{N-s}{N}} \leq P_{K}(E\setminus B_r(0)),
	\end{equation}
	where $C_8$ is a positive constant independent of $r$. Thus, combining \eqref{pfBound3} with \eqref{pfBound4}, we have
	\begin{equation}\label{pfBound5}
	C_8\,f(r)^{\frac{N-s}{N}} \leq C\,f(r) -\frac{\lambda\omega_{N-1}}{s}\int_{r}^{\infty}\frac{1}{(t-r)^s}f'(t)\,dt.
	\end{equation}
	Recalling the fact that $f(r)\to0$ as $r\to\infty$, we can choose $R_1>R_0$ such that
	\begin{equation}\label{pfBound6}
	C_9\,f(r) \leq \frac{C_9}{2}f(r)^{\frac{N-s}{N}}
	\end{equation}
	for any $r\geq R_1$. Therefore, for all $r\geq R_1$, we obtain
	\begin{equation}\label{pfBound7}
	\frac{s\,C_9}{\lambda\omega_{N-1}}f(r)^{\frac{N-s}{N}}\leq -\int_{r}^{\infty}\frac{1}{(t-r)^s}f'(t)\,dt.
	\end{equation}
	We integrate over $(R,\,\infty)$ where $R\geq R_1$ and then we exchange the order of integration to obtain 
	\begin{align}\label{pfBound8}
	\frac{s\,C_9}{\lambda_1\omega_{N-1}}\int_{R}^{\infty}f(r)^{\frac{N-s}{N}}\,dr&\leq -\int_{R}^{\infty}\int_{r}^{\infty}\frac{1}{(t-r)^s}f'(t)\,dt\,dr\nonumber\\
	&= - \int_{R}^{\infty}\int_{R}^{t}\frac{1}{(t-r)^s}f'(t)\,dr\,dt = -\frac{1}{1-s}\int_{R}^{\infty}f'(t)(t-R)^{1-s}\,dt. 
	\end{align}
	Moreover, recalling $f(r)\to0$ as $r\to\infty$, we have that
	\begin{align}\label{pfBound9}
	-\int_{R}^{\infty}f'(t)(t-R)^{1-s}\,dt&= -\int_{R}^{R+1}f'(t)(t-R)^{1-s}\,dt- \int_{R+1}^{\infty}f'(t)(t-R)^{1-s}\,dt \nonumber\\
	&\leq f(R) - f(R+1) - \int_{R+1}^{\infty}f'(t)(t-R)^{1-s}\,dt \nonumber\\
	&\leq  f(R) + \int_{R+1}^{\infty}f'(t)(1-(t-R)^{1-s})\,dt \nonumber\\
	&\leq  f(R) + (1-s)\int_{R+1}^{\infty}f(t)(t-R)^{-s}\,dt\leq f(R)+\int_{R}^{\infty}f(t)\,dt.
	\end{align}
	Thus, substituting \eqref{pfBound9} with \eqref{pfBound8}, we obtain
	\begin{equation}\label{pfBound10}
	\frac{s\,C_9}{\lambda\omega_{N-1}}\int_{R}^{\infty}f(r)^{\frac{N-s}{N}}\,dr \leq f(R)+\int_{R}^{\infty}f(r)\,dr.
	\end{equation}
	Since $f(r)$ is small for sufficiently large $R>R_1$, we may assume 
	\begin{equation}\label{pfBound11}
	2\int_{R}^{\infty}f(r)\,dr \leq \frac{s\,C_9}{\lambda\omega_{N-1}}\int_{R}^{\infty}f(r)^{\frac{N-s}{N}}\,dr
	\end{equation}
	for any $R>R_2$ and some $R_2>R_1$. Therefore, we conclude that, for any $R>R_2$,
	\begin{equation}\label{pfBound12}
	C_{10}\int_{R}^{\infty}f(r)^{\frac{N-s}{N}}\,dr \leq f(R),
	\end{equation}
	where $C_{10}=C_{10}(N,\,s,\,\lambda):=\frac{s\,C_9}{2\lambda\omega_{N-1}}>0$.
	
	Let $R>R_2$ be fixed such that $w_0=|E\setminus B_{R}(0)|>0$ is sufficiently small. For any $k\in\mathbb{Z}$ with $k\geq 0$, we set $\alpha:=\frac{N-s}{N}$, $R_k:=R+1-2^{-k}$, and $w_k:=f(R_k)$. Then, from \eqref{pfBound12}, we have that $R_k\to R_{\infty}:=R+1$ as $k\to\infty$ and
	\begin{equation}\label{pfBound13}
	C_{10}\,2^{-(k+1)}w_{k+1}^{\alpha} \leq w_k
	\end{equation}
	for any $k$. Then, by iterating this estimate and recalling that $w_0$ can be chosen sufficiently small, we obtain that $w_k \to 0$ as $k\to\infty$. However, by the assumption, we also have that $\lim_{k\to\infty}w_{k}=f(R+1)=|E\setminus B_{R+1}(0)|>0$, which is a contradiction.
\end{proof}

\section{Nonexistence of minimizers for $\mathcal{F}_{(K,\,A)}$}\label{nonexisMinimizer}
In this section, we show the proof of Theorem \ref{mainTheorem}. First of all, we define the quantity
\begin{equation}\label{minimumFunctional}
	F_{(K,\,A)}[m]:= \inf\{\mathcal{F}_{(K,\,A)}(E)\mid|E|=m,\,E:\text{bounded}\}
\end{equation}
for a given number $m\in(0,\,\infty)$. Because of the last term of the functional \eqref{nonlocalFunc2}, we cannot expect the subadditivity of the functional $\mathcal{F}_{(K,\,A)}$. Moreover, if we decompose the Riesz potential $V_1(E)$ into two parts $V_1(E_1)$ and $V_1(E_2)$ where $E=E_1\cup E_2$, then we usually gain the energy while the nonlocal perimeter $P_K$ is not the case. This also implies the invalidity of the subadditivity. However, since we can move the two bounded sets far away from each other to decrease the extra potential energy arising from the decomposition, we may have a weak version of the subadditivity in terms of the quantity \eqref{minimumFunctional} as follows:
\begin{lemma}\label{weakSubadditive}
	Suppose that $K$ satifies (K1) and (K2). Let $m_1$ and $m_2$ be a positive number. Then, it holds that
	\begin{equation}\label{subadditivity}
		F_{(K,\,A)}[m_1+m_2] \leq  F_{(K,\,A)}[m_1] + F_{(K,\,0)}[m_2].
	\end{equation}
\end{lemma}
\begin{proof}[Proof of Lemma \ref{weakSubadditive}]
	The proof can be done in a similar way with in \cite{luOtto}, and thus, we basically follow their strategy. Let $\varepsilon>0$ be an arbitrary number. Then, by the definition of \eqref{minimumFunctional}, there exist bounded subsets $E_1,\,E_2\subset\mathbb{R}^N$ with the volume constraints $|E_1|=m_1$ and $|E_2|=m_2$ such that
	\begin{equation}\label{pfLemma1}
		\mathcal{F}_{(K,\,A)}(E_1) + \mathcal{F}_{(K,\,0)}(E_2) \leq F_{(K,\,A)}[m_1]+ F_{(K,\,0)}[m_2] + \varepsilon.
	\end{equation}
	Since $E_1,\,E_2$ are bounded, we can find a sufficiently large number $d=d(\varepsilon)>0$ such that $\dist(E_1,\,(E_2+d\,e_1))\geq d/2$. Then, from \eqref{nonlocalPeriEquali} and \eqref{coulombPotEquali}, we may calculate as follows:
	\begin{align}\label{pfLemma2}
		\mathcal{F}_{(K,\,A)}(E_1\cup(E_2+d\,e_1)) &= P_K(E_1\cup(E_2+d\,e_1)) + V_1(E_1\cup(E_2+d\,e_1)) - A\,R(E_1\cup(E_2+d\,e_1) \nonumber\\
		&\leq P_K(E_1) + P_K(E_2+d\,e_1) + V_1(E_1) + V_1(E_2+d\,e_1) \nonumber\\
		&\quad \qquad+ \int_{E_1}\int_{E_2+d\,e_1}\frac{1}{|x-y|}\,dx\,dy - A\,R(E_1) \nonumber\\
		&\leq \mathcal{F}_{(K,\,A)}(E_1) + \mathcal{F}_{(K,\,0)}(E_2) + \frac{2m_1\,m_2}{d}.
	\end{align}
	Note that $P_K$ and $V_1$ is invariant under translations and $|x-y|\geq d/2$ for any $x\in E_1$ and $y\in E_2+d\,e_1$. Hence, by the definition of \eqref{minimumFunctional}, we obtain
	\begin{equation}\label{pfLemma3}
		F_{(K,\,A)}[m_1+m_2] \leq F_{(K,\,A)}[m_1]+ F_{(K,\,0)}[m_2] + \varepsilon+ \frac{2m_1\,m_2}{d}.
	\end{equation}
	Letting $d\to \infty$, and then $\varepsilon\to 0$, we conclude that the lemma holds.
\end{proof}

\begin{proof}[Proof of Theorem \ref{mainTheorem}]\label{proofMainThe}
	First of all, we prove the former claim of the main theorem. To do this, we assume that $K$ satisfies (K1), (K2), and (K4). We suppose that there exists a bounded minimizer $E\subset\mathbb{R}^N$ with $|E|=m$ of \eqref{minimumFunctional} for given $m$. Then we will show that $m$ actually satisfies the opposite ineqality to \eqref{volumeCondition}. 
	In order to divide $\mathbb{R}^N$ into two parts, we define the hyperplane $H_{\nu,\,l}$ by $H_{\nu,\,l}:=\{x\in\mathbb{R}^N\mid x\cdot\nu =l\}$ for any parameters $\nu\in\mathbb{S}^{N-1}$ and $l\in\mathbb{R}$. Moreover, we set 
	\begin{equation}\label{halfPlaneSet}
		H^+_{\nu,\,l}:=\{x\in\mathbb{R}\mid x\cdot \nu \geq l\},\quad H^{-}_{\nu,\,l}:=\mathbb{R}^{N-1}\setminus H^{+}_{\nu,\,l}.
	\end{equation}
	and
	\begin{equation}\label{restHalfPlane}
		E^{+}_{\nu,\,l}:=E\cap H^+_{\nu,\,l},\quad E^{-}_{\nu,\,l}:= E\cap H^{-}_{\nu,\,l}
	\end{equation}
	for any set $E\subset\mathbb{R}^{N}$ for any $\nu\in\mathbb{S}^{N-1}$ and $l\in\mathbb{R}$. Next, we want to compare the sum of the energies for $E^{+}_{\nu,\,l}$ and $E^{-}_{\nu,\,l}$ with the energy for $E$. To do this, we apply the Lemma \ref{weakSubadditive} and use the minimality of $E$ and then we have
	\begin{equation}\label{esti1}
		 \mathcal{F}_{(K,\,A)}(E) = F_{(K,\,A)}[m] \leq F_{(K,\,A)}[|E^{+}_{\nu,\,l}|] + F_{(K,\,0)}[|E^{-}_{\nu,\,l}|] \leq \mathcal{F}_{(K,\,A)}(E^{+}_{\nu,\,l}) + \mathcal{F}_{(K,\,0)}(E^{-}_{\nu,\,l}),
	\end{equation}
	Thus, it can be rewritten as
	\begin{equation}\label{esti2}
		P_K(E)+V_1(E)-A\,R(E) \leq P_K(E^{+}_{\nu,\,l})+V_1(E^{+}_{\nu,\,l})-A\,R(E^{+}_{\nu,\,l}) + P_K(E^{-}_{\nu,\,l})+V_1(E^{-}_{\nu,\,l}).
	\end{equation}
	Therefore, from \eqref{nonlocalPeriEquali} and \eqref{coulombPotEquali}, we obtain
	\begin{equation}\label{esti3}
		\int_{E^{+}_{\nu,\,l}}\int_{E^{-}_{\nu,\,l}}\frac{1}{|x-y|}\,dx\,dy \leq 2\int_{E^{+}_{\nu,\,l}}\int_{E^{-}_{\nu,\,l}} K(x-y)\,dx\,dy + A\,\int_{E^{-}_{\nu,\,l}}\frac{1}{|x|}\,dx.
	\end{equation}
	By the layer cake formula and Fubini's theorem, we may obtain the following integration result:
	\begin{align}\label{esti4}
		\int_{E^{-}_{\nu,\,0}}\frac{-x\cdot\nu}{|x|}\,dx &= \int_{E^{-}_{\nu,\,0}}\int_{-\infty}^{0}\frac{\chi_{(x\cdot\nu,\,0)}(l)}{|x|}\,dl\,dx \nonumber\\
		&= \int_{-\infty}^{0}\int_{E^{-}_{\nu,\,0}}\frac{\chi_{\{x\cdot\nu<l\}}(x)}{|x|}\,dx\,dl = \int_{-\infty}^{0}\int_{E^{-}_{\nu,\,l}}\frac{1}{|x|}\,dx\,dl.
	\end{align}
	Integrating the inequality \eqref{esti3} with respect to $l$ from $-\infty$ to 0 and substituting \eqref{esti4} for \eqref{esti3}, we have
	\begin{equation}\label{esti5}
		\int_{-\infty}^{0}\int_{E^{+}_{\nu,\,l}}\int_{E^{-}_{\nu,\,l}}\frac{1}{|x-y|}\,dx\,dy\,dl \leq 2\int_{-\infty}^{0}\int_{E^{+}_{\nu,\,l}}\int_{E^{-}_{\nu,\,l}} K(x-y)\,dx\,dy\,dl + A\,\int_{E^{-}_{\nu,\,0}}\frac{|x\cdot\nu|}{|x|}\,dx.
	\end{equation}
	By interchanging the role of $E^{+}_{\nu,\,l}$ and $E^{-}_{\nu,\,l}$ in the above calculations, we obtain
	\begin{align}\label{esti6}
		&\int_{0}^{+\infty}\int_{E^{+}_{\nu,\,l}}\int_{E^{-}_{\nu,\,l}}\frac{1}{|x-y|}\,dx\,dy\,dl \nonumber\\
		&\quad\qquad\leq 2\int_{0}^{+\infty}\int_{E^{+}_{\nu,\,l}}\int_{E^{-}_{\nu,\,l}} K(x-y)\,dx\,dy\,dl + A\,\int_{E^{+}_{\nu,\,0}}\frac{|x\cdot\nu|}{|x|}\,dx.
	\end{align}
	Thus, summing up \eqref{esti5} and \eqref{esti6}, we have that
	\begin{equation}\label{esti7}
		\int_{-\infty}^{+\infty}\int_{E^{+}_{\nu,\,l}}\int_{E^{-}_{\nu,\,l}}\frac{1}{|x-y|}\,dx\,dy\,dl \leq 2\int_{-\infty}^{+\infty}\int_{E^{+}_{\nu,\,l}}\int_{E^{-}_{\nu,\,l}} K(x-y)\,dx\,dy\,dl + A\,\int_{E}\frac{|x\cdot\nu|}{|x|}\,dx
	\end{equation} 
	Now, using the layer cake formula and Fubini's theorem again, we have that
	\begin{align}\label{esti8}
		\int_{-\infty}^{+\infty}\int_{E^{+}_{\nu,\,l}}\int_{E^{-}_{\nu,\,l}}\frac{1}{|x-y|}\,dx\,dy\,dl &= \int_{E}\int_{E}\int_{-\infty}^{+\infty}\frac{\chi_{\{x\cdot\nu<l\}}(x)\,\chi_{\{y\cdot\nu \geq l\}}(y)}{|x-y|}\,dl\,dx\,dy \nonumber\\
		&=\int_{E}\int_{E}\int_{-\infty}^{+\infty}\frac{\chi_{\{x\cdot\nu<l<y\cdot\nu\}}(l)}{|x-y|}\,dl\,dx\,dy \nonumber\\
		&= \int_{E}\int_{E} \frac{((y-x)\cdot\nu)_{+}}{|x-y|}\,dx\,dy.
	\end{align}
	For any fixed $x\in\mathbb{R}^{N}$, by the spherical polar coordinates with $x$ located on the $x_N$-axis, we obtain 
	\begin{equation}\label{sphereIntegral}
	\int_{\mathbb{S}^{N-1}}(x\cdot\nu)_{+}\,d\mathcal{H}^{N-1}(\nu) = \int_{\mathbb{S}^{N-2}}\int_{0}^{\frac{\pi}{2}} |x|\,\cos\theta\, d\theta\,d\mathcal{H}^{N-2} = \omega_{N-2}\,|x|.
	\end{equation}
	Since $\nu\in\mathbb{S}^{N-1}$ is any element, we have, by using Fubini's theorem again and \eqref{sphereIntegral}, that 
	\begin{align}\label{esti9}
		\int_{\mathbb{S}^{N-1}}\int_{E}\int_{E} \frac{((y-x)\cdot\nu)_{+}}{|x-y|}\,dx\,dy\,d\mathcal{H}^{N-1}(\nu) &= \int_{E}\int_{E} \int_{\mathbb{S}^{N-1}}\frac{((y-x)\cdot\nu)_{+}}{|x-y|}\,d\mathcal{H}^{N-1}(\nu)\,dx\,dy\nonumber\\
		&= \omega_{N-2}\,|E|^2_N.
	\end{align}
	Moreover, by Fubini's theorem and \eqref{sphereIntegral}, we also obtain
	\begin{equation}\label{esti10}
		\int_{\mathbb{S}^{N-1}}\int_{E}\frac{|x\cdot\nu|}{|x|}\,dx\,d\mathcal{H}^{N-1}(\nu) = 2\omega_{N-2}\,|E|.
	\end{equation}
	Now we consider the third term in \eqref{esti3}. By the three assumptions on $K$, and applying the calculation in \eqref{esti8} and Fubini's theorem, we may also compute as follows:
	\begin{align}\label{esti11}
		&\int_{\mathbb{S}^{N-1}}\int_{-\infty}^{+\infty}\int_{E^{+}_{\nu,\,l}}\int_{E^{-}_{\nu,\,l}} K(x-y)\,dx\,dy\,dl\,d\mathcal{H}^{N-1}(\nu) \nonumber\\
		&\quad = \int_{E}\int_{E}\int_{\mathbb{S}^{N-1}}((y-x)\cdot\nu)_{+}\,K(x-y)\,d\mathcal{H}^{N-1}(\nu)dx\,dy \nonumber\\
		&\quad =\omega_{N-2}\,\int_{E}\int_{E}|x-y|\,K(x-y)\,dx\,dy. \nonumber\\
		&\quad \leq \omega_{N-2}\,\int_{E}\int_{B_{1+\varepsilon}(y)} |x-y|\,K(x-y)\,dx\,dy + \omega_{N-2}\,\int_{E}\int_{E\cap B^c_{1+\varepsilon}(y)} |x-y|\,K(x-y)\,dx\,dy \nonumber\\
		&\quad \leq \omega_{N-2}\, \int_{E}\int_{B_{1+\varepsilon}(y)}\frac{1}{|x-y|^{N+s-1}}\,dx\,dy + \omega_{N-2} \,\int_{E}\int_{E} \frac{1}{(1+\varepsilon)^{N+s-1}}\,dx\,dy \nonumber\\
		& \quad = \omega_{N-2}\int_{\mathbb{S}^{N-1}}\int_{0}^{1+\varepsilon}\frac{1}{r^s}\,dr\,d\mathcal{H}^{N-1} + \frac{\omega_{N-2}}{(1+\varepsilon)^{N+s-1}}|E|^2_N \nonumber\\
		&\quad =  \frac{\omega_{N-2}\, \omega_{N-1}}{1-s}(1+\varepsilon)^{1-s} |E| + \frac{\omega_{N-2}}{(1+\varepsilon)^{N+s-1}}|E|^2_N.
	\end{align}
	Therefore, integrating the both sides in \eqref{esti7} with respect to $\nu$ in $\mathbb{S}^{N-1}$ and substituting \eqref{esti9}, \eqref{esti10}, and \eqref{esti11}, we obtain
	\begin{equation}\label{esti12}
		\omega_{N-2}\,|E|^2_N\leq  \frac{2\omega_{N-2}\, \omega_{N-1}}{1-s}(1+\varepsilon)^{1-s} |E| + \frac{2\omega_{N-2}}{(1+\varepsilon)^{N+s-1}}|E|^2_N + 2A\omega_{N-2}\,|E|.
	\end{equation}
	Hence, recalling $|E|=m$, we conclude that
	\begin{equation}\label{esti13}
		\left(\frac{1}{2}-\frac{1}{(1+\varepsilon)^{N+s-1}}\right)\,m \leq \frac{\omega_{N-1}\,(1+\varepsilon)^{1-s}}{1-s} + A.
	\end{equation}
	Note that, by the definition of $\varepsilon$, it holds $(1+\varepsilon)^{N+s-1}>2$.
	
	The second claim of the main theorem is readily obtained in the following manner; if we assume that $K$ satisfies (K1), (K2), (K3), and (K4)', then, by Lemma \ref{mainLemma}, every minimizer of the functional $\mathcal{F}_{(K,\,A)}$ under the volume constraint is actually bounded. Then, applying the argument shown in the above, we may derive the same result.
\end{proof}

\section{Generalization of results}\label{appendix}
In this appendix, we slightly modify the energy \eqref{nonlocalFunc2} and consider the nonexistence of its minimizers. The modification energy of \eqref{nonlocalFunc2} denoted by $\mathcal{F}_{(K,\,A,\,\beta)}$ is defined as follows:
\begin{equation}\label{appendix1}
	\mathcal{F}_{(K,\,A,\,\beta)}(E):=P_{K}(E)+V_1(E)-A\,\int_{E}\frac{1}{|x|^\beta}\,dx,
\end{equation}
for any measurable $E\subset\mathbb{R}^N$ with a volume constraint, where $\beta\in[0,\,N+1)$. Note that, if $\beta=1$, then we may obtain the same results as discussed in the above. If $\beta>1$, then we may also obtain the similar results stated in our main theorem, although the explicit value of the critical mass cannot be obtained. Indeed, if $\beta\in[0,\,N+1)$, then, by the symmetric rearrangement inequality, it holds that
	\begin{equation}\label{extendResult1}
	\int_{E}\frac{1}{|x|^{\beta}}\,dx \leq \int_{E^*}\frac{1}{|x|^{\beta}}\,dx
	\end{equation}
	where $E^*$ is the symmetric rearrangement of any measurable set $E$ with a finite volume. Thus, setting $r_E>0$ as a radius of the open ball $E^*$ with $\omega_N\,(r_{E})^N=|E|=m$ and taking $\beta<N+1$, \eqref{esti7}, \eqref{sphereIntegral}, and \eqref{esti10} into consideration, we may compute as follows:
	\begin{align}\label{extendResult2}
	\int_{\mathbb{S}^{N-1}}\int_{E^*}\frac{|x\cdot\nu|}{|x|^{\beta}}\,dx\,d\mathcal{H}^{N-1}(\nu) &= 2\omega_{N-2}\,\int_{E^*}\frac{1}{|x|^{\beta-1}}\,dx \nonumber\\ &=2\omega_{N-2}\int_{\mathbb{S}^{N-1}}\,d\mathcal{H}^{N-1}\int_{0}^{r_E}\frac{1}{r^{\beta-1}}\,r^{N-1}\,dr\nonumber\\
	&= \frac{2\omega_{N-2}\omega_{N-1}}{(N+1-\beta)\,\omega_N^{1-\frac{\beta-1}{N}}}\,m^{1-\frac{\beta-1}{N}}.
	\end{align}
	Thus, following the argument of the proof of Theorem \ref{mainTheorem} as above and using \eqref{extendResult1} and \eqref{extendResult2}, we have that
	\begin{equation}\label{extendResult3}
	\left(\frac{1}{2}-\frac{1}{(1+\varepsilon)^{N+1-s}}\right)m^2\leq \frac{\omega_{N-1}}{1-s}(1+\varepsilon)^{1-s}\,m + A\, \frac{\omega_{N-1}}{N+1-\beta} \omega_N^{-1+\frac{\beta-1}{N}}\,m^{1-\frac{\beta-1}{N}}.
	\end{equation}
	Here we set $p:=\frac{\beta-1}{N}\in[-\frac{1}{N},\,1)$ and define constants $C_1,\,C_2,\,C_3>0$ as
	\begin{equation}\label{extendResult4}
	C_1:=\left(\frac{1}{2}-\frac{1}{(1+\varepsilon)^{N+1-s}}\right),\,C_2:=\frac{\omega_{N-1}}{1-s}(1+\varepsilon)^{1-s},\,C_3:=\frac{\omega_{N-1}}{N+1-\beta} \omega_N^{-1+\frac{\beta-1}{N}}.
	\end{equation}
	Now, considering a profile of the function $\phi(x):=C_1\,x^{1+p}-C_2\,x^p-A\,C_3$ for any $x>0$, we may obtain that there exists a unique number $m_p>0$ such that $\phi(m_p)=0$ and $\phi(m)>0$ for any $m>m_p$. Therefore, imposing the same assumptions as in Theorem \ref{mainTheorem} for instance, we have that there exists a critical mass $m_p>0$ such that, for any $m>m_p$, the problem $\min\{\mathcal{F}_{(K,\,A,\,\beta)}(E)\mid |E|=m\}$ has no solutions.

\section*{\centering Acknowledgments}
The author would like to thank Matteo Novaga (University of Pisa) for fruitful discussions and for advice.

\small{

}

\end{document}